\title{The Ramsey number of dense graphs}
\author{David Conlon\thanks{St John's College, Cambridge CB2 1TP, United Kingdom.
E-mail: {\tt D.Conlon@dpmms.cam.ac.uk}. Supported by a
research fellowship at St John's College.}}

\documentclass[11pt]{article}
\usepackage{amsfonts,amssymb,amsmath,latexsym}

\newenvironment{proof}
      {\medskip\noindent{\bf Proof.}\hspace{1mm}}
      {\hfill$\Box$\medskip}

\def\qed{\ifvmode\mbox{ }\else\unskip\fi\hskip 1em plus 10fill$\Box$}

\newtheorem{theorem}{Theorem}[section]

\newtheorem{lemma}[theorem]{Lemma}

\newtheorem{corollary}[theorem]{Corollary}

\newtheorem{problem}[theorem]{Problem}

\def\d{\delta}
\def\D{\Delta}
\def\e{\epsilon}

\def\r{\rho}
\def\s{\sigma}

\begin{document}
\date{}
\maketitle

\begin{abstract}
The Ramsey number $r(H)$ of a graph $H$ is the smallest number $n$ such that, in any two-colouring of the edges of $K_n$, there is a monochromatic copy of $H$. We study the Ramsey number of graphs $H$ with $t$ vertices and density $\r$, proving that $r(H) \leq 2^{c \sqrt{\r} \log (2/\r) t}$. We also investigate some related problems, such as the Ramsey number of graphs with $t$ vertices and maximum degree $\r t$ and the Ramsey number of random graphs in $\mathcal{G}(t, \r)$, that is, graphs on $t$ vertices where each edge has been chosen independently with probability $\r$.
\end{abstract}

\section{Introduction}

Given a graph $H$, the {\it Ramsey number $r(H)$} is defined to be the smallest natural number $n$ such that, in any two-colouring of the edges of $K_n$, there exists a monochromatic copy of $H$. That these numbers exist was first proven by Ramsey \cite{R30} and rediscovered independently by Erd\H{o}s and Szekeres \cite{ES35}. Since their time, and particularly since the 1970's, Ramsey theory has grown into one of the most active areas of research within combinatorics, overlapping variously with graph theory, number theory, geometry and logic. 

The most famous question in the field is that of estimating the Ramsey number $r(t)$ of the complete graph $K_t$ on $t$ vertices. Despite some small improvements \cite{C09, S75}, the standard estimates, that $\sqrt{2}^t \leq r(t) \leq 4^t$, have remained largely unchanged for over sixty years. What, however, happens if one takes a slightly less dense graph on $t$ vertices? One would expect, for example, that if $H$ is a graph with only half the edges of a complete graph then $r(H) \leq (4 - \e)^t$ for some positive $\e$. Curiously, no theorem of this variety seems to be known. Our aim is to bridge this apparent omission in the theory. 

The density of a graph $H$ with $t$ vertices and $m$ edges is given by $\r = m/\binom{t}{2}$. We would like to determine the Ramsey number of a graph $H$ with $t$ vertices and given density $\r$. We shall always assume that $H$ has no isolated vertices. Otherwise, we could have graphs with zero density and arbitrarily large Ramsey number.

To get a lower bound, consider a graph on $t$ vertices containing a clique with $\frac{\sqrt{\r}}{2} t$ vertices, with the remaining edges (around $\frac{3\r}{4} \binom{t}{2}$ of them) distributed so that the graph has no isolated vertices. By the usual lower bound on Ramsey numbers, we see that the Ramsey number of this graph is at least $2^{\sqrt{\r} t/4}$. We prove an upper bound which comes close to matching this lower bound. In particular, it gives an exponential improvement on the trivial bound $4^t$ when $\r$ is a fixed, though small, positive density.

\begin{theorem} \label{IntroMain}
There exists a constant $c$ such that any graph $H$ on $t$ vertices with density $\r$ satisfies 
\[r(H) \leq 2^{c \sqrt{\r} \log(2/\r) t}.\]
\end{theorem}

We shall also prove some related results. Given two graphs $H_1$ and $H_2$, the {\it Ramsey number $r(H_1, H_2)$} is the smallest natural number $n$ such that, in any red/blue-colouring of the edges of $K_n$, there is guaranteed to be a blue copy of $H_1$ or a red copy of $H_2$. The Ramsey number $r(K_t, H)$ of the complete graph $K_t$ against a graph $H$ with $t$ vertices and maximum degree $\r t$ turns out to be of particular importance. A method of Graham, R\"odl and Ruci\'nski \cite{GRR00} easily implies that $r(K_t, H) \leq 2^{c \r t \log^2 t}$. We replace the $\log^2 t$ factor with a similar factor depending only on $\r$.

\begin{theorem} \label{IntroCliqueMax}
There exists a constant $c$ such that any graph $H$ on $t$ vertices with maximum degree $\r t$ satisfies 
\[r(K_t, H) \leq 2^{c \r \log^2 (2/\r) t}.\]
\end{theorem}

As a corollary of this result, we can prove an upper bound for the Ramsey number of the complete graph $K_t$ against a graph $H$ with $t$ vertices and density $\r$. 

\begin{theorem} \label{IntroCliqueDense}
There exists a constant $c$ such that any graph $H$ on $t$ vertices with density $\r$ satisfies 
\[r(K_t, H) \leq 2^{c \sqrt{\r} \log^{3/2} (2/\r) t}.\]
\end{theorem}

Note that both of these bounds are already quite good. For Theorem \ref{IntroCliqueMax}, a random argument gives a lower bound of the form $r(K_t, H) \geq 2^{c \r \log(2/\r) t}$. For Theorem \ref{IntroCliqueDense}, note that the Ramsey number $r(K_t, K_{\sqrt{\r} t/2}) \geq 2^{c \sqrt{\r} \log(2/\r) t}$. If we now place the remaining edges to form a graph $H$ on $t$ vertices with density $\r$ and no isolated vertices, we have $r(K_t, H) \geq 2^{c \sqrt{\r} \log(2/\r) t}$.

A similar question to that we have been looking at, suggested by Erd\H{o}s \cite{E84}, is to determine the Ramsey number of a graph $H$ with a given number of edges. It is an elementary consequence of the standard bounds for $r(t)$ that if $m$ is the number of edges in the complete graph $K_t$ then $r(t) \leq 2^{c \sqrt{m}}$. Erd\H{o}s conjectured that a similar upper bound should hold for all graphs $H$, that is, he conjectured the existence of a constant $c$ such that if $H$ is any graph with $m$ edges then $r(H) \leq 2^{c \sqrt{m}}$. For bipartite graphs, this conjecture was verified by Alon, Krivelevich and Sudakov \cite{AKS03}. Furthermore, by using the machinery of Graham, R\"{o}dl and Ruci\'nski \cite{GRR00}, they made significant progress towards the full conjecture, showing that, for any graph $H$ with $m$ edges, $r(H) \leq 2^{c \sqrt{m} \log m}$.

If we substitute $m = \r \binom{t}{2}$ in the result of Alon, Krivelevich and Sudakov, we find that $r(H) \leq 2^{c \sqrt{\r} t \log t}$ for any graph $H$ with $t$ vertices and density $\r$. Theorem \ref{IntroMain} improves on this result. Moreover, putting $\r = m/\binom{t}{2}$ in Theorem \ref{IntroMain}, we get the following theorem.

\begin{theorem}
There exists a constant $c$ such that any graph $H$ with $m$ edges and density $\r$ satisfies
\[r(H) \leq 2^{c \sqrt{m} \log (2/\r)}.\]
\end{theorem}

In particular, since a graph with no isolated vertices satisfies $\r \geq \frac{1}{t} \geq \frac{1}{2m}$, we have another proof that $r(H) \leq 2^{c \sqrt{m} \log m}$ for any graph $H$ with $m$ edges. 

Our methods also allow us to study Ramsey numbers of dense random graphs. The {\it binomial random graph $\mathcal{G}(t, \r)$} is the probability space consisting of all labelled graphs on $t$ vertices where each edge is chosen independently with probability $\r$. We shall say that the random graph $\mathcal{G}(t,\r)$ possesses a graph property $\mathcal{P}$ {\it almost surely} if the probability that $\mathcal{G}(t,\r)$ satisfies $\mathcal{P}$ tends to $1$ as $t$ tends to infinity. For sparse random graphs, taken, for $d$ fixed, with probability $\r = d/t$, the Ramsey number of graphs $H \in \mathcal{G}(t,\r)$ was studied by Fox and Sudakov \cite{FS09}, who showed that, almost surely, $2^{c_1 d} t \leq r(H) \leq 2^{c_2 d \log^2 d} t$.

A first estimate for the Ramsey number of dense random graphs follows from Theorem \ref{IntroCliqueMax}. This theorem easily implies that if a graph $H$ on $t$ vertices has maximum degree at most $2 \r t$, then $r(H) \leq 2^{c \r \log^2(2/\r) t}$. But, provided $\r \geq c' \frac{\log t}{t}$, a random graph $H \in \mathcal{G}(t, \r)$ will almost surely have maximum degree at most $2 \r t$, from which it follows that $r(H) \leq 2^{c \r \log^2 (2/\r) t}$. For $\r$ large, we show how this may be improved still further.

\begin{theorem} \label{IntroRandom}
There exist constants $c$ and $c'$ such that, if $H \in \mathcal{G}(t, \r)$ is a random graph with $\r \geq c' \frac{\log^{3/2} t}{\sqrt{t}}$, $H$ almost surely satisfies
\[r(H) \leq 2^{c \r \log(2/\r) t}.\]
\end{theorem}

For the lower bound, note that graphs in $\mathcal{G}(t,\r)$, with $\r \geq \frac{1}{t}$, almost surely have at least $\frac{\r}{2} \binom{t}{2}$ edges. The usual random arguments now imply that the Ramsey number of a graph with this many edges is at least $2^{c \r t}$. So our results are again very close to being sharp. 

We will begin, in Section \ref{EmbeddingSection}, by discussing an embedding lemma, due to Graham, R\"odl and Ruci\'nski \cite{GRR00}, which will be a crucial component in all of our proofs. Roughly speaking, this lemma says that if the edges of a graph $G$ are well-distributed, in the sense that every two large bipartite graphs have at least a fixed positive density of edges between them, then $G$ contains a copy of any small graph $H$. In Section \ref{DenseSection}, we will prove Theorems \ref{IntroMain}, \ref{IntroCliqueMax} and \ref{IntroCliqueDense}. In Section \ref{RandomSection}, we prove Theorem \ref{IntroRandom}. We conclude with a number of open questions. Throughout the paper, we systematically omit floor and ceiling signs. We also do not make any serious attempt to optimize absolute constants in our statements and proofs. All logs, unless stated otherwise, are taken to the base 2.

\section{The embedding lemma} \label{EmbeddingSection}

Let $G$ be a graph on vertex set $V$ and let $X, Y$ be two subsets of $V$. Define $e(X,Y)$ to be the number of edges between $X$ and $Y$. The density of the pair $(X, Y)$ is
\[d(X,Y) = \frac{e(X,Y)}{|X||Y|}.\]
The graph $G$ is said to be {\it bi-$(\s, \d)$-dense} if, for all $X, Y \subset V$ with $X \cap Y = \emptyset$ and $|X|, |Y| \geq \s |V|$, we have $d(X,Y) \geq \d$. It was proven by Graham, R\"{o}dl and Ruci\'nski \cite{GRR00} that if $\s$ is sufficiently small depending on $\d$ and the maximum degree of a fixed graph $H$ then a sufficiently large bi-$(\s, \d)$-dense graph $G$ must contain a copy of $H$. For the sake of completeness, we include a proof of their embedding lemma. 

\begin{lemma} \label{Embedding}
Let $\d > 0$ be a real number. If $G$ is a bi-$(\frac{1}{4} \d^{\D} \D^{-2}, \d)$-dense graph on at least $4 \d^{-\D} \D n$ vertices then $G$ contains a copy of any graph $H$ on $n$ vertices with maximum degree $\D$. 
\end{lemma}

\begin{proof}
Let $V$ be the vertex set of $G$ and suppose without loss of generality that $|V| = (\D + 1) N$, where $N \geq 2 \d^{-\D} n$. Split $V$ into $\D + 1$ pieces $V_1, V_2, \cdots, V_{\D+1}$, each of size $N$. Since the chromatic number of $H$ is at most $\D + 1$, we may split its set of vertices $W$ into $\D + 1$ independent sets $W_1, W_2, \cdots, W_{\D + 1}$. We will give an embedding $f$ of $H$ in $G$ so that $f(W_i) \subset V_i$ for all $1 \leq i \leq \D + 1$.

Let the vertices of $H$ be $\{w_1, w_2, \cdots, w_n\}$. For each $1 \leq h \leq n$, let $L_h = \{w_1, w_2, \cdots, w_h\}$. For each $y \in W_j \char92 L_h$, let $T_y^h$ be the set of vertices in $V_j$ which are adjacent to all already embedded neighbours of $y$. That is, letting $N_h (y) = N(y) \cap L_h$, $T_y^h$ is the set of vertices in $V_j$ adjacent to each element of $f(N_h (y))$. We will find, by induction, an embedding of $L_h$ such that, for each $y \in W \char92 L_h$, $|T_y^h| \geq \d^{|N_h(y)|} N$. 

For $h = 0$, there is nothing to prove. We may therefore assume that $L_h$ has been embedded consistent with the induction hypothesis and attempt to embed $w = w_{h+1}$ into an appropriate $v \in T_w^h$. Let $Y$ be the set of neighbours of $w$ which are not yet embedded. We wish to find an element $v \in T_w^h \char92 f(L_h)$ such that, for all $y \in Y$, $|N(v) \cap T_y^h| \geq \d |T_y^h|$. If such a vertex $v$ exists, taking $f(w) = v$ will then complete the proof.

Let $B_y$ be the set of vertices in $T_w^h$ which are bad for $y \in Y$, that is, such that $|N(v) \cap T_y^h| < \d |T_y^h|$. Note that, by the induction hypothesis, $|T_y^h| \geq \d^{\D} N \geq \frac{1}{2} \d^\D \D^{-1} |V|$. Therefore, $|B_y| < \frac{1}{4} \d^{\D} \D^{-2} |V| \leq \frac{1}{2} \d^\D \D^{-1} N$, for otherwise the density between the sets $B_y$ and $T_y^h$ would be less than $\d$, contradicting the bi-density condition. Hence, since $N \geq 2 \d^{-\D} n$,
\[\left| T_w^h \char92 \cup_{y \in Y} B_y \right| > \d^\D N - \D \frac{1}{2} \d^\D \D^{-1} N \geq n.\]
Hence, since at most $n$ vertices have already been embedded, an appropriate choice for $f(w)$ exists. 
\end{proof}

\section{Dense graphs} \label{DenseSection}

We shall begin by proving Theorem \ref{IntroCliqueMax}. The two main ingredients in the proof are Lemma \ref{Embedding} and the observation, due to Erd\H{o}s and Szemer\'edi \cite{ES72}, that if one of the colours in a two-coloured graph is known to have high density then it must contain a much larger clique than one would normally expect. We will not actually apply the Erd\H{o}s-Szemer\'edi result directly, but the underlying moral of their result is crucial to the proof.

\begin{theorem} \label{CliqueMax}
Let $H$ be a graph on $t$ vertices with maximum degree $\r t$. Then, provided $\r \leq \frac{1}{16}$,
\[r(K_t, H) \leq 2^{12 \r \log^2(2/\r) t}.\]
\end{theorem}

\begin{proof}
We shall prove, by induction on $s$, that for $s \geq \r t$, 
\[r(K_s, H) \leq \left(\frac{2s}{\r t}\right)^{12 \r \log(2/\r) t}.\]
The result follows from taking $s = t$.

The base case, $s = \r t$, is easy, since 
\[r(K_s, H) \leq r(K_s, K_t) \leq \binom{s+t}{s} = \binom{(1+\r)t}{\r t} \leq \left(\frac{e(1+\r)}{\r}\right)^{\r t} \leq \left(\frac{2}{\r}\right)^{2 \r t}.\]
Suppose, therefore, that the result is true for all $s < s_0$ and we wish to prove it for $s_0$.

Let $G$ be a graph on 
$$N = \left(\frac{2 s_0}{\r t}\right)^{12 \r \log(2/\r) t}$$ 
vertices whose edges are two-coloured in red and blue. By Lemma \ref{Embedding} with $\d = \r$, if the red subgraph is bi-$(\frac{1}{4} \r^{\r t} t^{-2}, \r)$-dense and $N \geq 4 \r^{-\r t} t^2$, there is a  copy of $H$ in red. We may therefore assume otherwise. That is, there exist two sets $A$ and $B$, each of size at least $\frac{1}{4} \r^{\r t} t^{-2} N$, such that the density of red edges between $A$ and $B$ is less than $\r$. Looking at it another way, the density of blue edges between $A$ and $B$ is at least $1 - \r$. 

Note now that there exists $A' \subseteq A$ such that $|A'| \geq \r |A|$ and, for each $v \in A'$, the blue degree $d_{B} (v)$ of $v$ in $B$ is at least $(1 - 2 \r)|B|$. Suppose otherwise. Then the density of edges between $A$ and $B$ is less than
\[(1 - \r)(1 - 2 \r) + \r \leq 1 - \r,\]
a contradiction. 

Letting $P = 12 \r \log(2/\r) t$, note that, since $(3/2)^2 \geq 2$,
\begin{eqnarray*}
|A'| \geq \r |A| & \geq & \frac{1}{4} \r^{\r t + 1} t^{-2} \left(\frac{2 s_0}{\r t}\right)^{12 \r \log(2/\r) t}
\geq \frac{1}{4} \r^{\r t + 1} t^{-2} \left(\frac{2 s_0}{\frac{4}{3} s_0}\right)^P \left(\frac{\frac{4}{3} s_0}{\r t}\right)^P\\
& \geq & \frac{1}{4} \r^{\r t + 1} t^{-2} \left(\frac{3}{2}\right)^P \left(\frac{\frac{4}{3} s_0}{\r t}\right)^P
\geq \frac{1}{4} \r^{\r t + 1} t^{-2} \left( \frac{2}{\r} \right)^{6 \r t} \left(\frac{\frac{4}{3} s_0}{\r t}\right)^P\\
& \geq & \left(\frac{2}{\r}\right)^{2\r t} \left(\frac{2 \left(\frac{2}{3} s_0\right)}{\r t}\right)^{12 \r \log(2/\r) t}.
\end{eqnarray*}
The last line follows since, for $\frac{1}{t} \leq \r \leq \frac{1}{e}$ (the former being a necessary condition for the graph to have no isolated vertices), the function $\r^{-3 \r t + 1}$ is increasing and, therefore, the inequality $2^{4 \r t - 2} \r^{-3\r t+1} \geq t^2$ holds. Therefore, by induction, $A'$ contains either a blue clique on $\frac{2}{3} s_0$ vertices or a red copy of $H$. Note that the extra $(2/\r)^{2 \r t}$ factor is there to account for the possibility that $\frac{2}{3} s_0$ is smaller than $\r t$. We may assume that $A'$ contains a blue clique $S$ of size $\frac{2}{3} s_0$.

By choice, every element of $A'$ has blue degree at least $(1 - 2 \r)|B|$ in $B$. Hence, the blue density between $S$ and $B$ is at least $1 - 2\r$. Following the usual K\H{o}v\'ari-S\'os-Tur\'an argument \cite{KST54}, we count the number of blue copies of the bipartite graph $K_{1, l}$, where the single vertex lies in $B$ and the collection of $l$ vertices lies in $S$. If this set has size at least $\binom{|S|}{l} r(K_{s_0 - l}, H)$, we are done. To see this, note that the condition implies the existence of a blue $K_l$ all of whose vertices are joined, in blue edges, to every vertex in a set of size $r(K_{s_0 - l}, H)$. This latter set contains either a red copy of $H$, in which case we are done, or a blue $K_{s_0 - l}$. If we add this latter set to the blue $K_l$ we get a blue $K_{s_0}$, so we are again done.

Let $l = \frac{1}{2} s_0$. We are going to show that for this choice of $l$, the number of $K_{1,l}$ is at least $\binom{|S|}{l} r(K_{s_0 - l}, H)$. To prove this, let $d_S(v)$ be the degree of a vertex $v$ from $B$ in $S$. Note, by convexity, that the number of $K_{1, l}$ is at least 
\[\sum_{v \in B} \binom{d_S(v)}{l} \geq |B| \binom{\frac{1}{|B|} \sum_{v \in B} d_S(v)}{l}
\geq |B| \binom{(1-2\r)|S|}{l}.\]
Note that, since $|S| = \frac{2}{3} s_0$ and $l = \frac{1}{2} s_0$, we have $|S| - l = \frac{1}{6} s_0 = \frac{1}{4} |S|$. Therefore,
\[\binom{(1-2\r)|S|}{l}/\binom{|S|}{l} = \prod_{i=0}^{l-1} \left(\frac{(1-2\r)|S| - i}{|S| - i}\right) \geq \prod_{i=0}^{l-1} \left(1 - \frac{2 \r |S|}{|S|-i} \right) \geq (1 - 8 \r)^l.\]
Therefore, since $l \leq t$, $\r \leq \frac{1}{16}$ and, for $0 \leq x \leq \frac{1}{2}$, we have $1 - x \geq 2^{-2x}$, 
\begin{eqnarray*}
|B| \frac{\binom{(1-2\r)|S|}{l}}{\binom{|S|}{l}} & \geq & \frac{1}{4} \r^{\r t} t^{-2} (1 - 8\r)^t \left(\frac{2 s_0}{\r t}\right)^{12 \r \log(2/\r) t}\\
& \geq & \frac{1}{4} \r^{\r t} t^{-2} 2^{-16 \r t} \left(\frac{2}{\r}\right)^{12 \r t} \left(\frac{s_0}{\r t}\right)^{12 \r \log(2/\r) t}\\
& \geq & \left(\frac{2}{\r}\right)^{2 \r t} \left(\frac{2\left(\frac{1}{2}s_0\right)}{\r t}\right)^{12 \r \log(2/\r) t}.
\end{eqnarray*}  
The last line follows since, for $\frac{1}{t} \leq \r \leq \frac{1}{2 e}$, the function $(2 \r)^{-9 \r t}$ is increasing, and, therefore, the inequality $2^{-6 \r t - 2} \r^{-9\r t} \geq (2 \r)^{-9 \r t} \geq t^2$ holds. By the induction hypothesis and the fact that $s_0 - l = \frac{1}{2} s_0$, this is greater than $r(K_{s_0 - l}, H)$ and the theorem is therefore proven. 
\end{proof}

By appropriately adjusting a method of Alon, Krivelevich and Sudakov \cite{AKS03}, we may now prove Theorem \ref{IntroCliqueDense} as a corollary of Theorem \ref{CliqueMax}.

\begin{corollary} \label{CliqueDense}
Let $H$ be a graph on $t$ vertices with density $\r$. Then, provided $\r \leq \frac{1}{50}$,
\[r(K_t, H) \leq 2^{15 \sqrt{\r} \log^{3/2} (2/\r) t}.\]
\end{corollary}

\begin{proof}
There are at most $t \sqrt{\r \log(2/\r)}$ vertices in $H$ with degree greater than $t \sqrt{\r/\log(2/\r)}$. Otherwise, the graph would contain more than $\r \binom{t}{2}$ edges, which would be a contradiction. Let $H'$ be the graph formed from $H$ by removing these vertices. By choice, it has maximum degree at most $t \sqrt{\r/\log(2/\r)}$. 

Consider a complete graph on $N = 2^{15 \sqrt{\r} \log^{3/2} (2/\r) t}$ vertices whose edges have been two-coloured in red and blue. We will construct a sequence of subsets of this vertex set $U_1 \supset U_2 \supset \cdots \supset U_l$ and a string $S$ consisting of $R$s and $B$s associated with this choice. To begin, let $u_1$ be an arbitrary vertex. If $u_1$ has at least $\r N$ neighbours in red, let $U_1$ be this set of neighbours and initalise the string as $S = R$. If, on the other hand, $u_1$ has at least $(1 - \r) N$ neighbours in blue, let $U_1$ be this set of neighbours and initialise the string as $S = B$. Suppose now that we have chosen $U_i$. Fix an arbitrary vertex $u_{i+1}$ in $U_i$. If $u_{i+1}$ has at least $\r |U_i|$ neighbours in red within $U_i$, we let $U_{i+1}$ be this set of neighbours and append an $R$ to our string $S$. Otherwise, we let $U_{i+1}$ be the set of blue neighbours and append $B$ to the end of the string. 

We stop our process when the string contains either $t-1$ occurrences of $B$ or $t \sqrt{\r \log(2/\r)}$ occurrences of $R$. If the first case occurs, there are $t - 1$ vertices $u_{i_1}, u_{i_2}, \cdots, u_{i_{t-1}}$ connected to each other and every vertex in the final set $U_l$ by blue edges. So, provided $U_l$ is non-empty, we have a blue $K_t$. If the second case occurs, there are, similarly, $t \sqrt{\r \log(2/\r)}$ vertices connected to each other and every vertex in $U_l$ by red edges. Note that, since $\r \leq \frac{1}{2}$ and $1 - x \geq 2^{-2x}$ whenever $0 \leq x \leq \frac{1}{2}$, $U_l$ has size at least 
\[\r^{t \sqrt{\r \log(2/\r)}} (1 - \r)^t N \geq 2^{-\sqrt{\r} \log^{3/2} (2/\r) t} 2^{-2 \r t} N \geq 2^{-2 \sqrt{\r} \log^{3/2} (2/\r) t} N.\]
Therefore, $|U_l| \geq 2^{13 \sqrt{\r} \log^{3/2} (2/\r) t} \geq 2^{12 \sqrt{\r} \log^{3/2} (2/\r) t} + t$. Since $H'$ has maximum degree $t \sqrt{\r/\log(2/\r)}$ and $\sqrt{\r/\log(2/\r)} \leq \frac{1}{16}$ whenever $\r \leq \frac{1}{50}$, Theorem \ref{CliqueMax} now tells us that the vertex set $U_l$ must contain either a blue copy of $K_t$ or a red copy of $H'$. The extra $t$ is needed so as to account for the fact that $H'$, unlike $H$, may have some isolated vertices. The result follows by adjoining this copy of $H'$, if it occurs, to the red clique of size $t \sqrt{\r \log(2/\r)}$ which is connected to $U_l$ by red edges.
\end{proof}

Theorem \ref{IntroMain} may now be proved in essentially the same manner as Corollary \ref{CliqueDense}. 

\begin{corollary} \label{DenseAllRange}
Let $H$ be a graph on $t$ vertices with density $\r$. Then, provided $\r \leq \frac{1}{16}$,
\[r(H) \leq 2^{15 \sqrt{\r} \log (2/\r) t}.\]
\end{corollary}

\begin{proof}
There are at most $t \sqrt{\r} \log(2/\r)$ vertices in $H$ with degree greater than $t \sqrt{\r}/\log(2/\r)$. Let $H'$ be the graph formed from $H$ by removing these vertices. By choice, it has maximum degree at most $t \sqrt{\r}/\log(2/\r)$. 

Consider a complete graph on $N = 2^{15 \sqrt{\r} \log (2/\r) t}$ vertices whose edges have been two-coloured in red and blue. As in the proof of Corollary \ref{CliqueDense}, we construct a sequence of subsets of this vertex set $U_1 \supset U_2 \supset \cdots \supset U_l$ and a string $S$ consisting of $R$s and $B$s associated with this choice. To begin, let $u_1$ be an arbitrary vertex. If $u_1$ has at least $N/2$ neighbours in red, let $U_1$ be this set of neighbours and initalise the string as $S = R$. If, on the other hand, $u_1$ has at least $N/2$ neighbours in blue, let $U_1$ be this set of neighbours and initialise the string as $S = B$. Suppose now that we have chosen $U_i$. Fix an arbitrary vertex $u_{i+1}$ in $U_i$. If $u_{i+1}$ has at least $|U_i|/2$ neighbours in red within $U_i$, we let $U_{i+1}$ be this set of neighbours and append an $R$ to our string $S$. Otherwise, we let $U_{i+1}$ be the set of blue neighbours and append $B$ to the end of the string. 

We stop our process when the string contains $\sqrt{\r} \log(2/\r) t$ occurrences of either $R$ or $B$. In either case, there are $d = \sqrt{\r} \log(2/\r) t$ vertices $u_{i_1}, u_{i_2}, \cdots, u_{i_d}$ which are all connected to each other and every vertex in the final set $U_l$ in one particular colour. Suppose, without loss of generality, that this colour is red. Therefore, if $U_l$ contains a blue clique of size $t$ or a red copy of $H'$, we will be done.  

To see that this is indeed the case, note that 
\[|U_l| \geq 2^{-2 \sqrt{\r}\log(2/\r)t} N \geq 2^{13 \sqrt{\r} \log(2/\r) t} \geq 2^{12 \sqrt{\r} \log(2/\r) t} + t.\]
Since $H'$ has maximum degree $t \sqrt{\r}/\log(2/\r)$ and $\sqrt{\r}/\log(2/\r) \leq \frac{1}{16}$ whenever $\r \leq \frac{1}{16}$, Theorem \ref{CliqueMax} now tells us that the vertex set $U_l$ must contain either a blue copy of $K_t$ or a red copy of $H'$. The extra $t$ is needed so as to account for the fact that $H'$ may have some isolated vertices. The result follows by adjoining this copy of $H'$, if it occurs, to the red clique of size $t \sqrt{\r} \log(2/\r)$ which is connected to $U_l$ by red edges.
\end{proof}

As we noted in the introduction, substituting $\r = m/\binom{t}{2}$ and using the fact that, for graphs with no isolated vertices, $t \leq 2 m$, this yields another proof that $r(H) \leq 2^{c \sqrt{m} \log m}$ for graphs with $m$ edges. 

\section{Random graphs} \label{RandomSection}

In this section, we will prove Theorem \ref{IntroRandom}. A key component of our proofs is a lemma saying that one may partition a graph into two pieces of comparable size such that the maximum degree within each of the induced subgraphs is approximately half the maximum degree of the original graph. To prove this, we will need the following estimate for the upper tail of the binomial distribution.

\begin{lemma} \label{Chernoff}
Let $X$ be a random variable that is binomially distributed with parameters $n$ and $p$ and let $0 \leq \theta \leq 1$ be a real number. Then $\mathbb{P}[X \geq (1 + \theta) pn] \leq e^{-\theta^2 pn/4}$.
\end{lemma}

\begin{proof}
A bound for the upper tail follows from the standard Chernoff bound
\[\mathbb{P}[X \geq (1+\theta) pn] \leq \left(\frac{e^{\theta}}{(1+\theta)^{1+\theta}}\right)^{p n}.\]
If, for $0 \leq \theta \leq 1$, we can show that $(1 + \theta)^{1+\theta} \geq e^{\theta +  \theta^2/4}$, this bound becomes simply $e^{-\theta^2 pn/4}$. Taking logs to the base $e$, it is sufficient to show that
\[(1+\theta) \log(1+\theta) \geq \theta + \frac{\theta^2}{4}.\] 
This clearly holds for $\theta = 0$. It is therefore sufficient to show that in the range $0 \leq \theta \leq 1$ the derivative, $1 + \log(1 + \theta)$, of the left hand side is at least the derivative, $1 + \theta/2$, of the right hand side. Again, these two expressions are equal at $\theta = 0$, so it is sufficient to show that the second derivative, $1/(1+\theta)$, of the left hand side is greater than or equal to the derivative, $1/2$, of the right hand side. But this follows easily from the condition $0 \leq \theta \leq 1$.
\end{proof}

We are now ready to prove our partitioning lemma.

\begin{lemma} \label{JudPart}
Let $H$ be a graph on $t$ vertices with maximum degree $\d t$. Then, provided that $\d \geq 64 \log t/t$ and $t \geq 16$, there is a partition of the graph into two vertex sets $V_1$ and $V_2$ such that, for $i = 1, 2$,
\[\left||V_i| - \frac{t}{2}\right| \leq 2 \sqrt{t}\]
and the maximum degree of any vertex into each of the vertex sets $V_1$, $V_2$ is at most 
\[\frac{\d}{2} t + 2 \sqrt{\d t \log t}.\]
\end{lemma}

\begin{proof}
We will partition the graph randomly by choosing each vertex to be in $V_1$ independently with probability $1/2$. Applying Lemma \ref{Chernoff} with $p = 1/2$, $n = t$ and $\theta = 4/\sqrt{t}$ (which is less than $1$ for $t \geq 16$), we see that, for $i = 1, 2$,  
\[\mathbb{P} [|V_i| - t/2 \geq 2 \sqrt{t}] \leq e^{-\theta^2 p n/4} = e^{-2} < \frac{1}{4}.\]
Therefore, since $V_1$ and $V_2$ are complementary, $||V_i| - t/2| \leq 2 \sqrt{t}$ with probability at least $1/2$. Now, for each vertex $v$, let $D(v)$ be the set of neighbours of $v$ and $d(v)$ be the size of $D(v)$. Let $d_i(v)$ be the random variable whose value is the size of $D_i(v) = D(v) \cap V_i$.
By Lemma \ref{Chernoff} with $p = 1/2$, $n = d(v)$ and $\theta = 4\sqrt{\log t/d(v)}$, we see that, provided $d(v) \geq 16 \log t$ (this is the necessary condition to have $\theta \leq 1$),
\[\mathbb{P}[d_i(v) - d(v)/2 \geq 2 \sqrt{\d t \log t}] \leq e^{-\theta^2 p n/4} \leq e^{- 2 \log t} \leq \frac{1}{t^2}.\]
If, on the other hand, $d(v) < 16 \log t$, the condition $\d \geq 64 \log t/t$ automatically implies that $d_i(v) \leq d(v) \leq 2 \sqrt{\d t \log t}$.
Adding over all $v$ and $i = 1, 2$, we see that $d_i(v) \leq d(v)/2 + 2 \sqrt{\d t \log t}$ holds in all cases with probability at least $1 - \frac{2}{t} > 1/2$. Since also $||V_i| - t/2| \leq 2 \sqrt{t}$ with probability greater than $1/2$, the result follows.
\end{proof}

The key property of random graphs that we will need to make use of is that for every vertex set $V$ of a given size there can only be a few vertices which have greater than the expected degree within $V$. The following lemma is sufficient for our purposes.

\begin{lemma} \label{RandomProp}
Let $H$ be a random graph on $t$ vertices formed by taking each edge independently with probability $\r$ and let $0 < \d, \e \leq 1$ be real numbers. Then, provided $\r \geq \frac{24\log t}{\e^2 \d t}$, $H$ satisfies the following condition with probability $e^{-\log(e/\d) \d t}$. For every vertex set $V$ of size $\d t$, the number of vertices with more than $(1 + \e) \r \d t$ neighbours in $V$ is at most $\frac{12 \log(e/\d)}{\r \e^2}$.
\end{lemma}

\begin{proof}
To prove the bound we again make use of Lemma \ref{Chernoff}. Indeed, given a fixed set $V$ of size $\d t$ and a vertex $u$, the variable $X_u$ counting the number of edges between $u$ and $V$ is binomial with probability $\r$. Therefore, by the Chernoff bound, 
\[\mathbb{P} (X_u \geq (1+\e) \r \d t) \leq e^{-\e^2 \r \d t/4}.\]
Since the $X_u$ are independent, we also see that, for any vertices $u_1, \cdots, u_d$, the probability that $X_{u_i} \geq (1+\e) \r \d t$ for all $1 \leq i \leq d$ is at most $e^{-d \e^2 \r \d t/4}$. The expected number of pairs consisting of a vertex set $V$ of size $\d t$ and a set of $d$ vertices $u_1, u_2, \cdots, u_d$ such that, for each $1 \leq i \leq d$, the number of neighbours that $u_i$ has in $V$ is at least $(1+\e)\r \d t$ is therefore at most
\[\binom{t}{\d t} \binom{t}{d} e^{-d \e^2 \r \d t/4}.\]
Taking $d = \frac{12 \log(e/\d)}{\r \e^2} \leq \frac{\log(e/\d) \d t}{2 \log t}$, we see that this is at most
\begin{eqnarray*}
\left(\frac{e}{\d}\right)^{\d t} \left(\frac{et}{d}\right)^d e^{-d \e^2 \r \d t/4} & = & e^{\log(e/\d) \d t} e^{\log(et/d) d} e^{-3 \log(e/\d) \d t}\\
& = & e^{\log(e t/d) d} e^{-2 \log(e/\d) \d t} \leq e^{-\log(e/\d) \d t}.
\end{eqnarray*}
The result follows.
\end{proof}

We will now prove Theorem \ref{IntroRandom}. Before we begin, we need a definition. We shall say that a graph $H$ is $(\D, q)$-bounded if, apart from an exceptional set of at most $q$ vertices, the maximum degree of every vertex in $H$ is $\D$.

\begin{theorem} \label{Max}
Let $H$ be a random graph on $t$ vertices such that each edge is chosen independently with probability $\r$. Then, provided $2^{15} \frac{\log^{3/2} t}{\sqrt{t}} \leq \r \leq \frac{1}{100}$, $H$ almost surely satisfies
\[r(H) \leq 2^{1100 \r \log (2/\r) t}.\]
\end{theorem}

\begin{proof}
By Lemma \ref{RandomProp} with $\d \geq t^{-1/2}$ and $\e = 1$, a random graph $H$, where each edge is chosen with probability $\r \geq \frac{24 \log t}{\sqrt{t}}$, will, with probability at least $1 - e^{-\d t} \geq 1 - 1/t^2$, be such that, for every set $V$ of size $\d t$, the number of vertices with more than $2 \r \d t$ neighbours in $V$ is at most $\frac{12 \log(e/\d)}{\r} \leq \sqrt{t}$. Adding over all possible sizes of $V$ between $\sqrt{t}$ and $t$, we see that with probability at least $1 - 1/t$, the graph $H$ will be such that, for every set $V$ of size at least $\sqrt{t}$, there are at most $\sqrt{t}$ vertices which have more than $2 \r |V|$ neighbours in $V$. 

We would also like our graph to satisfy a certain maximum degree condition. To this end, note that by Lemma \ref{Chernoff} with $n = t$, $p = \r$ and $\theta = 4 \sqrt{\log t/\r t}$, the maximum degree $d(v)$ of any given vertex $v$ will be such that
\[\mathbb{P}[d(v) - \r t \geq 4 \sqrt{\r t \log t}] \leq e^{-4 \log t} \leq \frac{1}{t^2}.\] 
Therefore, with probability at least $1 - 1/t$, every vertex in $H$ has degree at most $\r t + 4 \sqrt{\r t \log t}$. 

Combining the last two paragraphs, we see that, with probability at least $1 - 2/t$, $H$ is a graph with maximum degree $\r t + 4 \sqrt{\r t \log t}$ such that, for every set $V$ of size at least $\sqrt{t}$, there are at most $\sqrt{t}$ vertices with more than $2 \r |V|$ neighbours in $V$. We will henceforth assume that $H$ is just such a graph.

We shall prove, by induction, that, for all pairs $(s_1, s_2)$ such that $\r t \leq s_1, s_2 \leq t$, if $H_1$ and $H_2$ are $(\D_i, q_i)$-bounded subgraphs of $H$ on $s_1$ and $s_2$ vertices respectively with 
\[\D_i = \r s_i \left(1  + \frac{\log(2t/s_i)}{\log(2/\r)} \right)\]
and
\[q_i = \sqrt{t},\] 
then
\[r(H_1, H_2) \leq 2^{500(s_1 + s_2) \r \log (2/\r)} \left(\frac{s_1 + s_2}{\r^2 t}\right)^{40 \r t}.\]
The theorem follows from taking $s_1 = s_2 = t$, noting that the maximum degree of $H$ satisfies
\[\r t +  4 \sqrt{\r t \log t} \leq \r t \left(1 + \frac{\log 2}{\log(2/\r)}\right),\]
whenever $\r \geq 36 \log^3 t/t$.

If $s_1 = \r t$ or $s_2 = \r t$ the result is easy, since
\[r(H_1, H_2) \leq r(K_{\r t}, K_t) \leq \binom{(1+\r) t}{\r t} \leq \left(\frac{e(1+\r)}{\r}\right)^{\r t} \leq \left(\frac{2}{\r}\right)^{2 \r t}.\]
Suppose, therefore, that $t_1 \geq t_2 \geq \r t$ and the result is true for all admissible pairs $(s_1, s_2)$ with $s_1 \leq t_1$ and $s_2 < t_2$ or $s_1 < t_1$ and $s_2 \leq t_2$. We wish to prove it for the pair $(t_1, t_2)$. 

To begin, we take care of the exceptional sets. Consider a complete graph on 
\[M = 2^{500 (t_1 + t_2) \r \log(2/\r)} \left(\frac{t_1 + t_2}{\r^2 t}\right)^{40 \r t}\] 
vertices whose edges have been two-coloured in red and blue. We construct a sequence of subsets of this vertex set $U_1 \supset U_2 \supset \cdots \supset U_l$ and a string $S$ consisting of $R$s and $B$s associated with this choice. To begin, let $u_1$ be an arbitrary vertex. If $u_1$ has at least $\r N$ neighbours in red, let $U_1$ be this set of neighbours and initalise the string as $S = R$. If, on the other hand, $u_1$ has at least $(1 - \r) N$ neighbours in blue, let $U_1$ be this set of neighbours and initialise the string as $S = B$. Suppose now that we have chosen $U_i$. Fix an arbitrary vertex $u_{i+1}$ in $U_i$. If $u_{i+1}$ has at least $\r |U_i|$ neighbours in red within $U_i$, we let $U_{i+1}$ be this set of neighbours and append an $R$ to our string $S$. Otherwise, we let $U_{i+1}$ be the set of blue neighbours and append $B$ to the end of the string. 

We stop our process when the string contains either $t_1-1$ occurrences of $B$ or $\sqrt{t} \leq \frac{\r}{\log(2/\r)} t$ occurrences of $R$. If the first case occurs, there are $t_1 - 1$ vertices $u_{i_1}, u_{i_2}, \cdots, u_{i_{t_1-1}}$ connected to each other and every vertex in the final set $U_l$ by blue edges. So, provided $U_l$ is non-empty, we have a blue $K_{t_1}$ and we are done. We therefore assume that the second case occurs and that there are $\sqrt{t}$ vertices which are connected to each other and every vertex in $U_l$ by red edges. Note that, since $t_1 \leq t$, $\r \leq \frac{1}{2}$ and $1 - x \geq 2^{-2x}$ whenever $0 \leq x \leq \frac{1}{2}$, $U_l$ has size at least 
\[\r^{\r t/\log(2/\r)} (1-\r)^{t} M \geq 2^{-3 \r t} M.\]

If we run the same process again, with the colours reversed, we may find a subset $W$ of $U_l$ and a set of vertices of size $\sqrt{t}$ such that all of these vertices are connected to each other and every element in $W$ by blue edges. Moreover, we may choose $W$ so that
\[|W| \geq 2^{-3 \r t} |U_l| \geq 2^{-6 \r t} M.\]
Note that, since $\frac{4}{5} t_1 + t_2 \leq \frac{9}{10} (t_1 + t_2)$,
\begin{eqnarray*}
2^{-500(t_1+t_2)\r \log(2/\r)} |W| & \geq & 2^{-6 \r t} \left(\frac{t_1 + t_2}{\r^2 t}\right)^{40 \r t} = 2^{-6 \r t} \left(\frac{t_1 + t_2}{\frac{4}{5} t_1 + t_2}\right)^{40 \r t} \left(\frac{\frac{4}{5} t_1 + t_2}{\r^2 t}\right)^{40 \r t}\\
& \geq & 2^{-6 \r t} (10/9)^{40 \r t} \left(\frac{\frac{4}{5} t_1 + t_2}{\r^2 t}\right)^{40 \r t} \geq \left(\frac{\frac{4}{5} t_1 + t_2}{\r^2 t}\right)^{40 \r t}.
\end{eqnarray*}
If we now let $H'_1$ and $H'_2$ be the graphs formed from $H_1$ and $H_2$ by removing the exceptional vertices, it will be sufficient to show that, in any two-colouring of the edges of $W$, there is a blue copy of $H'_1$ or a red copy of $H'_2$.

Let $G$ be a complete graph on 
\[N = 2^{500(t_1 + t_2) \r \log(2/\r)} \left(\frac{\frac{4}{5} t_1 + t_2}{\r^2 t}\right)^{40 \r t}\] 
vertices whose edges are two-coloured in red and blue. Note that the maximum degree of $H'_2$ is at most $2 \r t_2$. Therefore, by Lemma \ref{Embedding} with $\d = \r$, if the red subgraph is bi-$(\frac{1}{4} \r^{2 \r t_2} t_2^{-2}, \r)$-dense and $N \geq 4 \r^{- 2\r t_2} t_2^2$, there is a copy of $H'_2$ in red. We may therefore assume otherwise, that is, there exist two sets $A$ and $B$, each of size at least $\frac{1}{4} \r^{2 \r t_2} t_2^{-2} N$, such that the density of blue edges between $A$ and $B$ is at least $1 - \r$. Note that there exists $A' \subseteq A$ such that $|A'| \geq \r |A|$ and, for each $v \in A'$, the blue degree $d_B(v)$ of $v$ in $B$ is at least $(1 - 2 \r)|B|$. Otherwise, the density of edges between $A$ and $B$ would be less than $1 - \r$.

Applying Lemma \ref{JudPart}, we see that, provided $\r \geq 64 \log t_1/t_1$ and $t_1 \geq 16$ (both of which hold for $\r \geq 16 \sqrt{\log t/t}$), we may partition the vertex set of $H'_1$ into two vertex sets $V_1$ and $V_2$ so that, for $i= 1, 2$,
\[\left||V_i| - \frac{t_1}{2}\right| \leq 2 \sqrt{t_1}\]
and, since $\D_1$ is the maximum degree of $H'_1$, the maximum degree of any vertex of $H'_1$ into $H'_1[V_i]$ is at most 
\[\frac{\D_1}{2} + 2 \sqrt{\D_1 \log t_1}.\]
Assume, without loss of generality, that $|V_1| \leq |V_2|$. Note that, for $\r \geq 64/t$, $t_1 \geq 64$ and therefore $2 \sqrt{t_1} \leq t_1/4$. This implies that $t_1/4 \leq |V_i| \leq 3 t_1/4$. Also, by assumption, $\D_1 \leq \r t_1 + \frac{\log(2 t/t_1)}{\log(2/\r)} \r t_1$. Therefore, since this expression is at most $2 \r t_1$, the maximum degree $\D(H'_1[V_i])$ of $H'_1[V_i]$ satisfies
\begin{eqnarray*}
\D(H'_1[V_i]) & \leq & \frac{\r t_1}{2} +  \frac{1}{2} \frac{\log(2t/t_1)}{\log(2/\r)} \r t_1 + 2 \sqrt{2 \r t_1 \log t_1}\\
& \leq & \frac{\r t_1}{2} + \frac{\log(13t/6t_1)}{\log(2/\r)} \frac{\r t_1}{2} \leq \frac{\r t_1}{2} \left(1 + \frac{\log(13t/8|V_i|)}{\log(2/\r)} \right).
\end{eqnarray*}
The second line follows since, for $t_1 \geq \r t$ and $\r \geq 60 \log^{3/2} t/\sqrt{t}$, 
\[2 \sqrt{2 \r t_1 \log t_1} \leq 2 \sqrt{2 \r t_1 \log t} \leq \frac{\log (13/12)}{\log(2/\r)} \frac{\r t_1}{2}.\]  
Noting that $\frac{t_1}{2} \leq |V_i| + 2 \sqrt{t_1}$, we have
\begin{eqnarray*}
\D(H'_1[V_i]) & \leq & \r |V_i| \left(1 + \frac{\log(13t/8|V_i|)}{\log(2/\r)} \right) + 4 \r \sqrt{t_1}\\
& \leq & \r |V_i| \left(1 + \frac{\log(7t/4|V_i|)}{\log(2/\r)} \right),
\end{eqnarray*}
since, for $t_1 \geq \r t$ and $\r \geq 2^{15} \log^2 t/t$, 
\[4 \r \sqrt{t_1} \leq \r \frac{t_1}{4} \frac{\log(14/13)}{\log(2/\r)} \leq \r |V_i| \frac{\log(14/13)}{\log(2/\r)}.\]

We will now show that $A'$ contains either a red copy of $H'_2$ or a blue copy of $H'_1[V_1]$.
Indeed, note that
\begin{eqnarray*}
\left(\frac{\frac{4}{5} t_1 + t_2}{\r^2 t}\right)^{-40 \r t} |A'| & \geq & \frac{1}{4} \r^{2 \r t_2 + 1} t_2^{-2} 2^{500(t_1 + t_2) \r \log(2/\r)}\\
& \geq & \frac{\r}{4 t_2^2} 2^{123 \r \log(2/\r) t_1} 2^{500(\frac{3}{4} t_1 + t_2) \r \log(2/\r)}\\
& \geq & 2^{500(\frac{3}{4} t_1 + t_2) \r \log(2/\r)}.
\end{eqnarray*}
The final line follows since, for $\r \geq 2 \sqrt{\log t/t}$, we have $2^{\r \log(2/\r) t_1} \geq 4 t_2^2/\r$. Therefore, 
\[|A'| \geq 2^{500(\frac{3}{4} t_1 + t_2) \r \log(2/\r)} \left(\frac{\frac{4}{5} t_1 + t_2}{\r^2 t}\right)^{40 \r t} > 2^{500(\frac{3}{4} t_1 + t_2) \r \log(2/\r)} \left(\frac{\frac{3}{4} t_1 + t_2}{\r^2 t}\right)^{40 \r t}\]
By induction, since $|V_1| \leq \frac{3}{4} t_1$ and $H'_1[V_1]$ has maximum degree less than
\[\r |V_1| \left(1 + \frac{\log(2t/|V_1|)}{\log(2/\r)} \right)\]
the graph contains either a red copy of $H'_2$ or a blue copy of $H'_1[V_1]$. Note that if $\frac{3}{4} t_1$ is smaller than $\r t$, we are still fine, since $t_2 \geq \r t$ and, therefore,
\[\left(\frac{\frac{3}{4} t_1 + t_2}{\r^2 t}\right)^{40 \r t} \geq 2^{40 \r \log(1/\r) t} \geq r(K_{\r t}, K_t).\] 
If $A'$ contains a red copy of $H'_2$, we are done. We therefore assume that $A'$ contains at least one blue copy of $H'_1[V_1]$. Let the vertex set of this copy of $H'_1[V_1]$ be $S$.

By choice, every element of $A'$ has blue degree at least $(1 - 2 \r)|B|$ in $B$. Hence, the blue density between $S$ and $B$ is at least $1 - 2\r$. Let $l = \left(1 - \frac{\log(15/14)}{2 \log(2/\r)}\right) |S|$. We are going to count the number of blue copies of the bipartite graph $K_{1, l}$, where the single vertex lies in $B$ and the collection of $l$ vertices lies in $S$. 

Let $d_S(v)$ be the degree of a vertex $v$ from $B$ in $S$. Note, by convexity, that the number of $K_{1, l}$ is at least 
\[\sum_{v \in B} \binom{d_S(v)}{l} \geq |B| \binom{\frac{1}{|B|} \sum_{v \in B} d_S(v)}{l}
\geq |B| \binom{(1-2\r)|S|}{l}.\]
Note that $|S| - l = \frac{\log(15/14)}{2 \log(2/\r)} |S|$ and $\log(15/14) \geq 1/12$. Therefore,
\[\binom{(1-2\r)|S|}{l}/\binom{|S|}{l} = \prod_{i=0}^{l-1} \left(\frac{(1-2\r)|S| - i}{|S| - i}\right) \geq \prod_{i=0}^{l-1} \left(1 - \frac{2 \r |S|}{|S|-i} \right) \geq (1 - 48 \r \log(2/\r))^l.\]
Therefore, since $l \leq t_1$, $\r \leq \frac{1}{96}$ and, for $0 \leq x \leq \frac{1}{2}$, we have $1 - x \geq 2^{-2x}$,
\begin{eqnarray*}
\left(\frac{\frac{4}{5} t_1 + t_2}{\r^2 t}\right)^{-40 \r t} |B| \frac{\binom{(1-2\r)|S|}{l}}{\binom{|S|}{l}} & \geq & \frac{1}{4} \r^{2 \r t_2} t_2^{-2} (1 - 48\r\log(2/\r))^{t_1}  2^{500(t_1 + t_2) \r \log(2/\r)}\\
& \geq & \frac{1}{4} \r^{2\r t_2} t_2^{-2} 2^{-96 \r \log(2/\r) t_1} 2^{500(t_1 + t_2) \r \log(2/\r)}\\
& \geq & t_2^{-2} 2^{2 \r \log(2/\r) t_1 - 2} 2^{500(\frac{4}{5} t_1 + t_2) \r \log(2/\r)}\\
& \geq & 2^{500(\frac{4}{5} t_1 + t_2) \r \log(2/\r)}.
\end{eqnarray*}  
The last line follows since, for $\frac{1}{\sqrt{t}} \leq \r \leq \frac{1}{8}$, we have $2^{2 \r \log(2/\r) t - 2} \geq 2^{8 \sqrt{t} - 2} \geq t^2$. There is, therefore, some subset $T$ of $S$, with size $(1 - \frac{\log(15/14)}{2 \log(2/\r)}) |S|$ such that at least $2^{500 (\frac{4}{5} t_1 + t_2) \r \log(2/\r)} \left(\frac{\frac{4}{5} t_1 + t_2}{\r^2 t}\right)^{40 \r t}$ vertices in $B$ are connected to each element of $T$ in blue. The set $T$ contains a subgraph $K$ of $H'_1[V_1]$. Let $L$ be the induced subgraph of $H'_1$ on the complementary vertex set to $K$. Note that $L$ includes $H'_1[V_2]$ as a subgraph. By the choice of $H$, since, provided $\r \geq 96 \log t/\sqrt{t}$, 
\[|S \char92 T| \geq \frac{|S|}{24 \log(2/\r)} \geq \frac{t_1}{96 \log(2/\r)} \geq \frac{\r t}{96 \log(2/\r)} \geq \sqrt{t},\] 
there exist at most $\sqrt{t}$ exceptional vertices in $H'_1$ with more than $2 \r |S \char92 T|$ neighbours in $S \char92 T$. Moreover, since every vertex in $H'_1$ has degree at most
\[\r |V_2| \left(1  + \frac{\log(7t/4|V_2|)}{\log(2/\r)} \right)\]
in $H'_1[V_2]$ and $|V_1| \leq |V_2|$, the degree in $L$ of the non-exceptional vertices is at most
\begin{eqnarray*}
\r |V_2|\left(1 + \frac{\log(7t/4|V_2|)}{\log(2/\r)} \right) + 2 \r |S \char92 T| & \leq & \r |V_2| \left(1 + \frac{\log(7t/4|V_2|)}{\log(2/\r)} \right) + 2 \r |V_1| \frac{\log(15/14)}{2 \log(2/\r)}\\
& \leq & \r |V_2| \left(1 + \frac{\log(15t/8|V_2|)}{\log(2/\r)} \right).
\end{eqnarray*}
Note now that the vertex set $V(L)$ of $L$ has size
\[V(L) = |V_2| + |V_1| \frac{\log(15/14)}{2 \log(2/\r)} \leq |V_2| \left(1+ \frac{ \log(15/14)}{2 \log(2/\r)}\right) \leq \frac{16}{15} |V_2|.\]
Therefore, the degree in $L$ of the non-exceptional vertices is at most
\[\r |V(L)| \left(1 + \frac{\log(2t/|V(L)|)}{\log(2/\r)}\right).\]
Note also, since $|V_2| \leq \frac{3}{4} t_1$, that $|V(L)| \leq \frac{16}{15} |V_2| \leq \frac{4}{5} t_1$.
 
Since there are at most $\sqrt{t}$ exceptional vertices, we may apply the induction hypothesis with $s_1 = |V(L)|$ and $s_2 = t_2$ to conclude that
\[r(L, H'_2) \leq 2^{500(\frac{4}{5} t_1 + t_2) \r \log(2/\r)} \left(\frac{\frac{4}{5} t_1 + t_2}{\r^2 t}\right)^{40 \r t}.\]
Now, recall that there is a set $B'$ of vertices in $B$ such that
\[|B'| \geq 2^{500(\frac{4}{5} t_1 + t_2) \r \log(2/\r)} \left(\frac{\frac{4}{5} t_1 + t_2}{\r^2 t}\right)^{40 \r t}\]
and, for every $x \in T$ and $y \in B'$, the edge $xy$ is blue. Therefore, $B'$ contains either a blue copy of $L$ or a red copy of $H'_2$. In the latter case we are done, so we assume that there is a blue copy of $L$. But the set $T$, to which it is joined exclusively by blue edges, contains a blue copy of its complement $K$ in $H'_1$. We therefore have a blue copy of $H'_1$.
\end{proof}

\section{Conclusion} \label{Conclusion}

The most obvious question that remains open is the first we asked. That is, given a graph $H$ on $t$ vertices with density $1/2$, show that its Ramsey number is smaller than $(4 - \e)^t$ for some positive $\e$. More generally, we have the following problem.

\begin{problem}
For every $\d > 0$, show that there is $\e > 0$ such that any graph $H$ of density $1 - \d$ satisfies $r(H) \leq (4 - \e)^t$.
\end{problem}

A second problem, which affects much of the recent work in Ramsey theory, is that our proofs do not extend to cover the multicolour case. We fully believe that Theorems \ref{IntroMain} and \ref{IntroRandom} should extend to the multicolour case, but it seems that a fundamentally new idea will be necessary to get anywhere near.

\begin{problem}
Extend Theorems \ref{IntroMain} and \ref{IntroRandom} to the multicolour case.
\end{problem}

It is natural also to try and suggest that the results of this paper be generalised to hypergraphs. It seems unlikely, however, that such a generalisation is true. Indeed, in \cite{CFS09}, the author, together with Fox and Sudakov, has shown that there are 3-uniform hypergraphs on $n$ vertices whose density tends to zero as $n$ gets large but whose 4-colour Ramsey number is at least $2^{2^{c n}}$. Up to the constant $c$, this is the same as the Ramsey number of the complete graph on $n$ vertices. While this does not rule out some sort of miracle in the 2-colour case, it does seem to make the possibility unlikely.

\end{document}